\documentclass[english,letterpaper,11pt,reqno]{amsart}

\usepackage{appendix}
\usepackage{amsbsy}
\usepackage{amsfonts}
\usepackage{amsmath}
\usepackage{amssymb}
\usepackage{amsthm}
\usepackage{graphicx}
\usepackage{ifthen}
\usepackage{textcomp}
\usepackage{enumitem, color, amssymb}
\usepackage{dsfont}
\usepackage{mathrsfs}
\usepackage{calrsfs}
\usepackage[bookmarksnumbered,colorlinks]{hyperref}
\hypersetup{colorlinks=true, linkcolor=blue, citecolor=blue}
\usepackage{soul}

\newtheorem{thm}{Theorem}

\newtheorem{prop}{Proposition}

\newcommand{\beqa}{\begin{eqnarray}}
\newcommand{\beq}{\begin{equation}}
\newcommand{\eeqa}{\end{eqnarray}}
\newcommand{\eeq}{\end{equation}}

\newcommand\ip[2]{g({#1},{#2})}

\newcommand{\inc}{\ensuremath{\lhook\joinrel\relbar\joinrel\rightarrow}}
\newcommand{\RR}{\mathbb{R}}
\newcommand{\vep}{\varepsilon}
\newcommand{\uu}{\mathcal{U}}

\newcommand\kk{{\boldsymbol k}}
\newcommand\KK{\tilde{{\boldsymbol k}}}
\newcommand\LL{{\boldsymbol \ell}}
\newcommand\vv[1]{{\boldsymbol {\it #1}}}
\newcommand\xx{\vv{x}}
\newcommand\yy{\vv{y}}
\newcommand\cd[2]{\nabla_{\!#1}{#2}}

\begin{document}
\title[]{Symplectic 4-manifolds via Lorentzian geometry}
\author[]{Amir Babak Aazami}
\address{Kavli IPMU (WPI), UTIAS\hfill\break\indent
The University of Tokyo\hfill\break\indent
Kashiwa, Chiba 277-8583, Japan}
\email{amir.aazami@ipmu.jp}

\maketitle
\begin{abstract}
We observe that, in dimension four, symplectic forms may be obtained via Lorentzian geometry; in particular, null vector fields can give rise to exact symplectic forms.  That a null vector field is nowhere vanishing yet orthogonal to itself is essential to this construction.  Specifically, we show that on a Lorentzian 4-manifold $(M,g)$, if $\kk$ is a complete null vector field with geodesic flow along which $\text{Ric}(\kk,\kk) > 0$, and if $f$ is any smooth function on $M$ with $\kk(f)$ nowhere vanishing, then $dg(e^f\kk,\cdot)$ is a symplectic form and $\kk/\kk(f)$ is a Liouville vector field; any null surface to which $\kk$ is tangent is then a Lagrangian submanifold.  Even if the Ricci curvature condition is not satisfied, one can still construct such symplectic forms with additional information from $\kk$; we give an example of this, with $\kk$ a complete Liouville vector field, on the maximally extended ``rapidly rotating" Kerr spacetime.  
\end{abstract}

\section{Introduction}
\label{sec:Intro}
The goal of this paper is twofold: to promote the use of Lorentzian geometry to construct Liouville manifolds in dimension four, and to motivate the use of symplectic techniques in the study of Lorentzian geometry.  To begin with, a \emph{symplectic form} on an even-dimensional smooth manifold $M$ is a closed nondegenerate 2-form.  A \emph{Lorentzian metric} on $M$ is a symmetric, nondegenerate 2-tensor with signature $(-+\cdots +)$.  Unlike Riemannian metrics, these yield \emph{null} vectors, which are nonzero but orthogonal to themselves.  Among their applications in general relativity, where they model the paths of light rays, we find a distinctly mathematical one for them here: they can give rise to exact symplectic forms on Lorentzian 4-manifolds, in much the same way that ``twisted" vector fields can give rise to contact forms in dimension three: 

\begin{thm}
\label{thm:f}
Let $(M,g)$ be a Lorentzian 4-manifold and $\kk$ a complete null vector field on $M$ satisfying $\cd{\kk}{\kk} = 0$ and \emph{$\text{Ric}(\kk,\kk) > 0$}.  If there exists a smooth function $f$ on $M$ such that $\kk(f)$ is nowhere vanishing, then $dg(e^f\kk,\cdot)$ is a symplectic form on $M$ and $\kk/\kk(f)$ is a Liouville vector field.
\end{thm}

Here $d$ is the exterior derivative and $g(\kk,\cdot)$ is the 1-form metrically equivalent to $\kk$.  The proof of Theorem \ref{thm:f}, which appears in Section \ref{sec:null} below, is a standard application of two well known equations from Lorentzian geometry, one of which is the \emph{Raychaudhuri equation} (see \cite[Prop. 5.7.2]{o1995}); the bottom line is that positive Ricci curvature along $\kk$ precludes its normal subbundle $\kk^{\perp} \subset TM$ from being integrable, so that the flow of $\kk$ is necessarily \emph{``\,twisted.''}  Together with a function $f$ such that $\kk(f)$ is nowhere vanishing, this is enough to ensure that the closed 2-form $dg(e^f\kk,\cdot)$ is nondegenerate.  In Section \ref{sec:null} we also provide two examples to show that: (1) positive Ricci curvature is a sufficient, but by no means necessary, condition to ensure twistedness; (2) the assumption of the completeness of $\kk$ cannot be dropped from Theorem \ref{thm:f}.

\vskip 12pt
Let us make a few more remarks about integrability here.  Recall that in three dimensions, the integrability of a subbundle normal to a vector field completely determines whether its corresponding 1-form is a \emph{contact form} (see \cite[Prop. 3.7.15, p. 178]{thurston}), where we recall that a contact form on an odd-dimensional smooth manifold $M$ is a 1-form $\theta$ such that at each point $p \in M$, $d\theta_p$ is nondegenerate on $\text{Ker}\,\theta_p \subset T_pM$.  Indeed, the symplectic form $dg(e^f\kk,\cdot)$ in Theorem \ref{thm:f} certainly resembles a symplectization of the 1-form $g(\kk,\cdot)$.  In Section \ref{sec:null}, we will provide an example of a symplectic form, constructed as in Theorem \ref{thm:f}, on the Lorentzian 4-manifold $(\RR \times \mathbb{S}^3,-dt^2\oplus \mathring{g})$, where $\mathring{g}$ is the round metric on the 3-sphere $\mathbb{S}^3$; this symplectic form, it turns out, \emph{will} be the symplectization of a contact form on $\mathbb{S}^3$ whose general form was first discovered in \cite{hp13}.  Specifically, \cite{hp13} showed that if $\kk$ is a unit vector field on a Riemannian 3-manifold $(M^3,g)$ satisfying $\cd{\kk}{\kk} = 0$ and $\text{Ric}(\kk,\kk) > 0$, then the 1-form $g(\kk,\cdot)$ is a contact form; this was then re-derived, by another means, in \cite{AA14}.  \emph{Theorem \ref{thm:f} above is essentially a four-dimensional symplectic version of the construction in \cite{AA14},} made possible for the following reason: because a null vector field $\kk$ uniquely satisfies $\kk \subset \kk^{\perp}$, one can thus consider the two-dimensional quotient subbundle $\kk^{\perp}/\kk$ \emph{instead} of the full three-dimensional subbundle $\kk^{\perp}$\,---\,this is the crucial (and well known) fact that ultimately makes Theorem \ref{thm:f} possible.  Regarding the existence of the function $f$, there is a well known class of Lorentzian 4-manifolds, namely, the \emph{globally hyperbolic} ones, which possess \emph{Cauchy temporal functions} $f$ as defined in \cite{bernal03,muller}, which naturally satisfy the property that $\kk(f)$ is nowhere vanishing.  These 4-manifolds split diffeomorphically as $\RR \times S$.  Finally, it is also worth noting that the existence of the function $f$ is also satisfied in any \emph{stably causal} spacetime, by choosing $f$ to be merely a \emph{temporal function} \cite{bernal03}; i.e., one whose level sets are not necessarily Cauchy hypersurfaces, as they are for Cauchy temporal functions (stably causal spacetimes comprise a strictly larger class of Lorentzian 4-manifolds than globally hyperbolic ones; see \cite{ming}).  
\vskip 12pt

Before proceeding to the proof of Theorem \ref{thm:f}, we point out that our method of proof also yields a result in three dimensions, which is of interest solely for Lorentzian geometry.  Recall that the three-dimensional Weinstein conjecture, proved by C. H. Taubes \cite{taubes07}, states that on a closed, oriented 3-manifold, the \emph{Reeb vector field} of any contact form has an integral curve that is \emph{closed,} where we recall that the Reeb vector field of a contact form $\theta$ is the (uniquely defined) smooth vector field $X$ satisfying $X \lrcorner\, d\theta = 0$ and $\theta(X) = 1$  Using \cite{taubes07}, in Section \ref{sec:null} we will show that the following is then true:

\begin{prop}
\label{prop:contact}
Let $(M,g)$ be a closed Lorentzian 3-manifold.  If $\kk$ is a constant length timelike vector field on $M$ satisfying $\cd{\kk}{\kk} = 0$ and \emph{$\text{Ric}(\kk,\kk) > 0$}, then one of its integral curves is closed.
\end{prop}

This result is relevant because, regarding closed geodesics on compact manifolds, in fact the Lorentzian setting has seen considerably less progress than the Riemannian; in particular, the question whether a compact Lorentzian manifold contains closed geodesics (of whatever causal character) is still open for dimensions $\geq 3$.  Nevertheless, there are some well known results, which we briefly summarize here.  Perhaps the first such was \cite{tipler79}.  It was shown therein that, if a compact Lorentzian manifold has a (regular) covering with a compact Cauchy surface, then it contains a closed timelike geodesic.   In the case when the covering is not compact, more recent results by \cite{guediri02,guediri07} and \cite{sanchez06} established the existence of closed timelike geodesics by assuming instead certain conditions on the group of deck transformations.  Next, \cite{galloway84} considered free timelike homotopies and showed that when a certain stability condition was obeyed, such homotopies necessarily contain a (longest) closed timelike geodesic.  Homotopy results have also been used when a compact Lorentzian manifold contains a hypersurface-orthogonal timelike Killing vector field (i.e., when the manifold is static); in particular, \cite{sanchez06} showed that such manifolds always contain closed timelike geodesics.  This result was then strengthened in \cite{flores11}, wherein it was shown that closed timelike geodesics exist even when the Killing vector field is not hypersurface-orthogonal or everywhere timelike.  All of these results hold for dimensions $\geq 2$.  In the case of only two dimensions, stronger results can be had.  Indeed, every compact Lorentzian 2-manifold contains a closed timelike or null geodesic \cite{galloway86}; in fact, it was shown in \cite{suhr13} that there must be at least two closed geodesics.

\section{Proofs of Results}
\label{sec:null}

\begin{proof}[Proof of Theorem \ref{thm:f}]
We have, by assumption, a triple $(M,g,\kk)$ as in Theorem \ref{thm:f}, as well as a smooth function $f$ on $M$ such that $\kk(f)$ is nowhere vanishing.  That $\kk$ is null and has geodesic flow means that the following two equations are obeyed by $\kk$ (see \cite[Proposition 5.7.2, p. 330]{o1995}),
\beqa
\kk(\text{div}\,\kk) &=& \frac{\iota^2}{2} - 2|\sigma|^2 - \frac{(\text{div}\,\kk)^2}{2} - \text{Ric}(\kk,\kk),\label{eqn:ray1}\\
\kk(\iota^2) &=& -2(\text{div}\,\kk)\,\iota^2,\label{eqn:ray2}
\eeqa
where $|\sigma|^2$ is the squared magnitude of the \emph{complex shear} associated to $\kk$'s flow, while the function $\iota^2$ vanishes at a point if and only if the normal subbundle $\kk^{\perp} \subset TM$ is integrable at that point; the latter follows from Frobenius's theorem.  Note that although $\sigma$ and $\iota$ are usually defined via local frames, both $|\sigma|^2$ and $\iota^2$ are in fact globally defined smooth functions on $M$; see \cite[p. 327ff.]{o1995}, bearing in mind that our $\iota$ is twice that of \cite{o1995} (in the literature ``$\iota$" is usually designated by ``$\omega$", but we will reserve the latter symbol for denoting symplectic forms; for a discussion of the significance of the quotient subbundle $\kk^{\perp}/\kk$ in the derivation of \eqref{eqn:ray1} and \eqref{eqn:ray2}, which we emphasized in the Introduction above, see \cite[p. 327ff.]{o1995}).  With that said, the usual argument now applies.  Namely, let $\gamma$ be an arbitrary geodesic integral curve of $\kk$, and suppose that $\iota^2 \circ \gamma = 0$.  Then given $\text{Ric}(\kk,\kk) > 0$, \eqref{eqn:ray1} simplifies along $\gamma$ to the inequality
\beqa
\label{eqn:ref}
\kk(\text{div}\,\kk) \circ \gamma\ <\ -\frac{(\text{div}\,\kk \circ \gamma)^2}{2},
\eeqa
which implies, because $\gamma$ is complete, that $\text{div}\,\kk \circ \gamma > 0$.  But by the same argument, it follows that $\text{div}(-\kk) \circ \gamma > 0$, a contradiction.  Therefore $\iota^2 \circ \gamma$ cannot be identically zero, in which case \eqref{eqn:ray2} implies that $\iota^2$ is in fact nowhere zero along $\gamma$.  As $\gamma$ was arbitrary, we thus have that $\iota^2$ is nowhere vanishing on $M$, hence that $\kk^{\perp} \subset TM$ is nowhere integrable on $M$.  Now consider the closed 2-form
\beqa
\label{eqn:w}
\omega\ :=\ d\ip{e^f\kk}{\cdot},
\eeqa
where $d$ is the exterior derivative and $f$ is as above.  If this is to be a symplectic form on $M$, then it must be nondegenerate.  Let $\{\kk,\xx,\yy,\LL\}$ be a local frame, with $\xx$ and $\yy$ spacelike orthonormal vectors orthogonal to $\kk$ and $\LL$, and $\LL$ a null vector field satisfying $g(\kk,\LL) = -1$ (consult \cite[p.~321]{o1995} for more on such frames, in terms of which $\iota = \ip{\cd{\yy}{\kk}}{\xx} - \ip{\cd{\xx}{\kk}}{\yy}$).  Then a computation shows that
\beqa
\omega(\kk,\LL) &=& -e^f \kk(f),\label{eqn:deg1}\\
\omega(\xx,\yy) &=& -e^f\iota,\label{eqn:deg2}
\eeqa
both of which are nowhere zero, so that $\omega$ must be nondegenerate (in fact $\text{det}\,\omega = e^{4f}(\kk(f))^2\,\iota^2$).
Finally, noting that $\mathscr{L}_{X}\omega = d(X\, \lrcorner\, \omega)$ because $\omega$ is closed, the vector field
$$
X\ :=\ \frac{\kk}{\kk(f)}
$$
satisfies $X\, \lrcorner\, \omega = \ip{e^{f}\kk}{\cdot}$, from which $\mathscr{L}_{X}\omega = \omega$ follows.  Therefore $(M,\omega)$ is an exact symplectic 4-manifold with Liouville vector field $X$.
\end{proof}

{\bf Remark 1.} By Frobenius's theorem, $\kk$ cannot be tangent to a null hypersurface (i.e., a codimension one embedded submanifold whose induced metric is degenerate), precisely because $\iota$ is nowhere vanishing.  But any null surface $S$ to which $\kk$ is tangent is necessarily \emph{Lagrangian,} $\omega|_S = 0$, because $\omega(\kk,\xx) = 0$ for any spacelike $\xx$ orthogonal to $\kk$.
\vskip 6pt
{\bf Remark 2.}  Given the form of \eqref{eqn:ray1}, it is clear first of all that the curvature assumption $\text{Ric}(\kk,\kk) > 0$ can be weakened to requiring only that $\text{Ric}(\kk,\kk) \geq 0$ but positive at some point on \emph{each} integral curve of $\kk$.  Having said that, neither $\text{Ric}(\kk,\kk) > 0$ nor this weakened version is strictly necessary to construct symplectic forms like $\omega$ in \eqref{eqn:w}.  Indeed, we now construct such an $\omega$ on an open subset of (maximally extended) Kerr spacetime, which is a \emph{Ricci flat} 4-manifold (all properties of the Kerr metric appearing below can be found in \cite{o1995}).  First, recall that the \emph{maximally extended ``rapidly rotating" Kerr spacetime} $(M,g)$, with mass $m > 0$ and angular momentum per unit mass $a > m$, is an open subset of the 4-manifold $\RR^2 \times \mathbb{S}^2$ with global coordinates $(t,r,\vartheta,\varphi)$; the Lorentzian metric $g$ is given in these coordinates by
\beqa
g_{tt}\ =\ -1 + \frac{2mr}{\rho^2}\hspace{.2in},\hspace{.2in}g_{rr} &=& \frac{\rho^2}{\Delta}\hspace{.2in},\hspace{.2in}g_{\vartheta\vartheta}\ =\ \rho^2,\nonumber\\
g_{\varphi\varphi}\ =\ \left[r^2 +a^2+\frac{2m r a^2 \sin^2\vartheta}{\rho^2}\right]\sin^2\vartheta&,&g_{\vartheta t}\ =\ g_{t \vartheta}\ =\ -\frac{2mra\sin^2\vartheta}{\rho},\nonumber
\eeqa
all other components being zero, and with
$$
\rho^2\ :=\ r^2 + a^2\cos^2\vartheta\hspace{.2in},\hspace{.2in}\Delta\ :=\ r^2 -2mr + a^2.
$$
Here both $r$ and $t$ take values on the entire real line $\RR$, while $0 \leq \vartheta \leq \pi$ and $0 \leq \varphi < 2\pi$ are coordinates on $\mathbb{S}^2$.  The metric $g$ can be smoothly extended over both the \emph{horizon} $\Delta = 0$ and the \emph{axis} $\sin\vartheta = 0$ (in fact, analytically so), but the so-called \emph{ring singularity} $\Sigma$ defined by $\rho^2 = 0$ (i.e., the set of points defined by $r = \cos \vartheta = 0$, which is topologically $\RR \times \mathbb{S}^1$) is a genuine curvature singularity (see \cite[Corollary 2.7.7, p. 101]{o1995}); in fact, the maximally extended rapidly rotating Kerr spacetime $M$ is precisely $\RR^2 \times \mathbb{S}^2 - \Sigma$ (the extension is achieved via so-called \emph{Kerr-star coordinates} $(t^*,r,\vartheta,\varphi^*)$, whose definition can be found in \cite[p. 80ff.]{o1995}).  By choosing $a > m$, observe that $\Delta > 0$ has no real roots; we may then take as our null vector field the so-called \emph{outgoing principal null vector field}
$$
\kk\ :=\ \partial_r + \frac{r^2 +a^2}{\Delta}\,\partial_t + \frac{a}{\Delta}\,\partial_{\vartheta}.
$$
That $\kk$ is null, future-pointing, and has geodesic flow is verified in \cite[p. 79ff.]{o1995}.  The corresponding function $\iota^2$ can be computed directly (see \cite[p. 331]{o1995} for a derivation), and is given by
\beqa
\label{eqn:pcv2}
\iota^2\ =\ \frac{4a^2\cos^2\vartheta}{\rho^4}\cdot
\eeqa
Taking $f$ to be our coordinate function $r$ (so that, in particular, $\kk(r) = 1$) and restricting to the ``northern hemisphere" of $M$, i.e., the open subset
$$
\uu\ :=\ \{(t,r,\vartheta,\varphi) \in \RR^2 \times \mathbb{S}^2 - \Sigma~:~0 \leq \vartheta < \pi/2\},
$$
wherein $\rho^2$ and $\cos^2\vartheta$ in \eqref{eqn:pcv2} are both nowhere vanishing, it follows from inspection of \eqref{eqn:deg1} and \eqref{eqn:deg2} in Theorem \ref{thm:f} that the exact 2-form
$$
\omega\ :=\ dg(e^r\kk,\cdot)
$$
is a symplectic form on $\uu \subset M$, and that $\kk|_{\uu}$ is a Liouville vector field of $(\uu,\omega|_{\uu})$.  In fact $\kk|_{\uu}$ is \emph{complete,} as follows.  Observe that for any $p \in \uu$, the integral curve $\gamma$ of $\kk$ starting at $\gamma(0) = p$ never leaves $\uu$, for if it did then $\iota^2 \circ \gamma$ would have to vanish at $\vartheta = \pi/2$, contradicting \eqref{eqn:ray2}, since $(\iota^2 \circ \gamma)(0) \neq 0$.  Then, because any geodesic in $M$ that does not ``hit" the ring singularity $\Sigma$ is complete, and because the only integral curves of $\kk$ that hit the ring singularity are those that lie on the equatorial plane $\vartheta = \pi/2$ (see \cite[Definition 2.7.6, p. 101 \& Theorem 4.3.1, p. 189]{o1995}), it follows that $\kk|_{\uu}$ has complete flow in $\uu$.
Note that, e.g., $dg(e^t\kk,\cdot)$ is also a symplectic form on $\uu$, but $\kk|_\uu$ is not a Liouville vector field for it.  Note also that we could just as well have worked on the ``southern hemisphere" of $M$, with $\pi/2 < \vartheta \leq \pi$.
In any event, we conclude that $(\uu,\omega|_{\uu})$ is a symplectic 4-manifold for which $\kk|_{\uu}$ is a (complete) Liouville vector field.  Observe that all the conditions for Theorem \ref{thm:f} hold for $(\uu,g|_{\uu},\kk|_{\uu},r|_{\uu})$, except for $\text{Ric}(\kk|_{\uu},\kk|_{\uu}) > 0$, or its weakened version above.
\vskip 6pt
{\bf Remark 3.} Having said that, we now show that the completeness of $\kk$ cannot be dropped from among the assumptions of Theorem \ref{thm:f}.  Indeed, consider the Lorentzian metric $\tilde{g} := e^{2r}g$, with $g$ the Kerr metric defined above and $M = \RR^2 \times \mathbb{S}^2 - \Sigma$, and set
$$
\KK\ :=\ e^{-2r}\kk.
$$
In the Lorentzian 4-manifold $(M,\tilde{g})$, it is straightforward to verify that the null vector field $\KK$ satisfies
$$
\widetilde{\nabla}_{\!\KK}\KK\ =\ 0\hspace{.2in},\hspace{.2in}\text{Ric}_{\tilde{g}}(\KK,\KK)\ =\ 2(dr(\KK))^2\ =\ 2e^{-4r}\ > \ 0,
$$ 
where $\widetilde{\nabla}$ is the Levi-Civita connection and $\text{Ric}_{\tilde{g}}$ the Ricci tensor of $(M,\tilde{g})$ (see, e.g., \cite[p. 59]{besse}).  However, any integral curve $\tilde{\gamma}$ of $\KK$ that lies on the equatorial plane $\vartheta = \pi/2$ cannot be complete; indeed, if any such $\tilde{\gamma}$ were complete, then precisely the same analysis of \eqref{eqn:ray1} and \eqref{eqn:ray2} in Theorem \ref{thm:f} would dictate that the normal subbundle $\KK^{\perp} \subset TM$ cannot be integrable along $\tilde{\gamma}$.  But $\KK^{\perp} = \kk^{\perp}$ and, by \eqref{eqn:pcv2}, the latter \emph{is} integrable on the equatorial plane.  We conclude that $(M,\tilde{g},\KK,r)$ is an example for which all the conditions for Theorem \ref{thm:f} hold except for the completeness of $\KK$, but for which the conclusions of Theorem \ref{thm:f} do not hold (for any $f$).
\vskip 6pt
Speaking more generally, suppose that the null vector field $\kk$ in Theorem \ref{thm:f} is not complete; then one may pick a complete Riemannian metric $g_R$ on $M$ and work with the complete vector field $\KK := \kk/|\kk|_{g_R}$ instead (observe that with respect to our original Lorentzian metric $g$, $\KK$ is null and satisfies $\text{Ric}(\tilde{\kk},\tilde{\kk}) > 0$).  The difficulty now is that $\KK$ in general has only \emph{pregeodesic} flow; i.e., setting $h := 1/|\kk|_{g_R}$, we have
$
\cd{\KK}{\KK} = \kk(h)\,\KK.
$
Now set $\psi := \text{div}\,\kk - \kk(h) $.  The analogues of \eqref{eqn:ray1} and \eqref{eqn:ray2} for $\KK$ are (for a derivation, which we forego here, modify the derivations of \eqref{eqn:ray1} and \eqref{eqn:ray2} in \cite[Prop.~5.8.9,~p.~339]{o1995}, noting that while $\kappa = 0$ as before, now $\vep + \bar{\vep} = \kk(h)$ and $2\rho = -\psi + i\,\tilde{\iota}$ therein)
\beqa
\KK(\psi) &=& \frac{\tilde{\iota}^{\,2}}{2} - 2|\tilde{\sigma}|^2 - \frac{\psi^2}{2} + \kk(h)\,\psi - \text{Ric}(\tilde{\kk},\tilde{\kk}),\label{eqn:ray3}\\
\KK(\tilde{\iota}^{\,2}) &=& -2(\psi - \kk(h))\,\tilde{\iota}^{\,2},\label{eqn:ray4}\nonumber
\eeqa
and where we have introduced tildes to distinguish between the corresponding functions for $\kk$ in \eqref{eqn:ray1} and \eqref{eqn:ray2} above (in fact $\tilde{\iota}^{\,2} = h^2\,\iota^2$ and $|\tilde{\sigma}|^2 = h^2\,|\tilde{\sigma}|^2$).
Unfortunately, \eqref{eqn:ray3} does not permit the same analysis as \eqref{eqn:ray1} afforded, unless the bound on the Ricci term is modified appropriately, e.g., by stipulating that $\text{Ric}(\kk,\kk) > \kk(h)\,(\text{div}\,\kk)$.

\vskip 6pt
{\bf Remark 4.} Let us make one more remark about completeness.  Let $\kk$ be a geodesic null vector field in a Lorentzian 4-manifold $(M,g)$.  If $(M,g)$ is a stably causal spacetime, which means that it is equipped with a smooth function $f$ on $M$ whose gradient is everywhere past-pointing timelike, then by \cite{beem} $(M,g)$ is conformal to a null geodesically complete Lorentzian 4-manifold $(M,\tilde{g})$, with $\tilde{g} = e^{2u}g$ for some smooth function $u$ on $M$.  It follows that $\KK := e^{-2u}\kk$ is a geodesically complete null vector field in $(M,\tilde{g})$.  In other words, in order to use Theorem \ref{thm:f} to construct a symplectic form on a smooth 4-manifold $M$ on which there exist a stably causal Lorentzian metric $g$ and a geodesic null vector field $\kk$, in principle only the curvature condition $\text{Ric}_{\tilde{g}}(\KK,\KK) > 0$ need be verified (for the properties of null vector fields under conformal transformations, consult \cite{candela}).
\vskip 6pt
{\bf Remark 5.}  Finally, we construct a null vector field on the Lorentzian 4-manifold $(\RR \times \mathbb{S}^3,-dt^2 \oplus \mathring{g})$ satisfying the conditions of Theorem \ref{thm:f}, where $(\mathbb{S}^3,\mathring{g})$ is the (Riemannian) round 3-sphere.  In coordinates $(x^1,y^1,x^2,y^2) \in \RR^4$, define the vector field
$$
\kk\ :=\ \sum_{i}-y^i \frac{\partial}{\partial x^i} + x^i \frac{\partial}{\partial y^i}\cdot
$$
In the usual local parametrization of, say, the upper hemisphere $(y^2 > 0)$ of $\mathbb{S}^3$, and assuming that the sphere has radius 1,
$$
(x^1,y^1,x^2) \inc \left(\,x^1,y^1,x^2,\sqrt{1-(x^1)^2-(y^1)^2-(x^2)^2}\,\right),
$$
the round metric $\mathring{g}$ is of the form
$$
\mathring{g}\ =\ (dx^1)^2 + (dy^1)^2 + (dx^2)^2 + \left(\frac{x^1 dx^1 + y^1 dy^1 + x^2 dx^2}{\sqrt{1-(x^1)^2-(y^1)^2-(x^2)^2}}\right)^{\!2},
$$
while $\kk$, which is tangent to $\mathbb{S}^3$, takes the form
\beqa
\label{eqn:K}
\kk\ =\ -y^1\frac{\partial}{\partial x^1} + x^1\frac{\partial}{\partial y^1} - \sqrt{1-(x^1)^2-(y^1)^2-(x^2)^2}\,\frac{\partial}{\partial x^2}\cdot
\eeqa
It is straightforward to verify that $\kk$ is a unit length Killing vector field in $(\mathbb{S}^3,\mathring{g})$, so that the flow of $\kk$ is, among other things, geodesic (being also divergence-free, it follows by \cite{gluck} that $\kk$ is tangent to the Hopf fibration).  Furthermore, $\text{Ric}_{\mathring{g}}(\kk,\kk) = 2$.  Consider now the Lorentzian metric $\tilde{g}$ on $\mathbb{S}^3$ defined by
\beqa
\label{eqn:lorentz}
\tilde{g}\ :=\ \mathring{g} - 2\mathring{g}(\kk,\cdot) \otimes \mathring{g}(\kk,\cdot),
\eeqa
with corresponding Levi-Civita connection $\widetilde{\nabla}$.  In $(\mathbb{S}^3,\tilde{g})$, $\kk$ is unit timelike, $\tilde{g}(\kk,\kk) = -1$, and in fact still a Killing vector field.  Now we construct a null vector field on the Lorentzian 4-manifold $(\RR \times \mathbb{S}^3,-dt^2 \oplus \mathring{g})$ satisfying the conditions of Theorem \ref{thm:f}.  In fact the desired null vector field is simply $\tilde{\kk} := d/dt + \kk$.  Setting $\tilde{g} := -dt^2 \oplus \mathring{g}$ and, by abuse of notation, denoting by $\widetilde{\nabla}$ the corresponding Levi-Civita connection, it follows that $\widetilde{\nabla}_{\!\tilde{\kk}}\tilde{\kk} = 0$.  Furthermore, $\tilde{\kk}$ is complete and satisfies $\text{Ric}_{\tilde{g}}(\tilde{\kk},\tilde{\kk}) = \text{Ric}_{\mathring{g}}(\kk,\kk) > 0$.  Taking as our function $f$ the projection $t$ of any point onto its $t$-coordinate, it follows that $d\tilde{g}(e^t\tilde{\kk},\cdot)$ is an exact symplectic form on $(\RR \times \mathbb{S}^3,\tilde{g})$.  Note that this is a symplectization of the contact form $\mathring{g}(\kk,\cdot)$ on $\mathbb{S}^3$.

\vskip 12pt
Finally, we close with a proof of Proposition \ref{prop:contact}:

\begin{proof}[Proof of Proposition \ref{prop:contact}]
Since $M$ is compact, the flow is complete and $\text{Ric}(\kk,\kk) \geq b$ for some positive constant $b$.  The proof now proceeds virtually identically to \cite[Corollary~1]{AA14}.  We thus have that $g(\kk,\cdot)$ is a contact form, hence so is $-g(\kk,\cdot)$, and $\kk$ will be the Reeb vector field of the latter (if $\kk$ did not have constant length, then $-g(\kk,\cdot)$ would still be a contact form by the same proof as in \cite[Corollary~1]{AA14}, but $\kk$ would not be its Reeb vector field).  In three dimensions, the Weinstein conjecture \cite{taubes07} states that every Reeb vector field on $M$ has an integral curve that is closed, so the proof is complete.
\end{proof}

{\bf Remark 6.}  On $(\mathbb{S}^3,\tilde{g})$, with Lorentzian metric $\tilde{g}$ given by \eqref{eqn:lorentz}, the Killing vector field $\kk$ given by \eqref{eqn:K} necessarily has a closed integral curve, by Proposition \ref{prop:contact}.  Having said that, a stronger result is known whenever a compact Lorentzian manifold possesses a timelike Killing vector field: there necessarily exists another timelike Killing vector field \emph{all} of whose integral curves are closed; see \cite{rsanchez}.  Therefore, although the triple $(\mathbb{S}^3,\tilde{g},\kk)$ in Remark 5 above is sufficient to illustrate Proposition \ref{prop:contact}, nevertheless the latter is better served in cases when $\kk$ is not a timelike Killing vector field.

\section*{Acknowledgements}
This work was supported by the World Premier International Research Center Initiative (WPI), MEXT, Japan.  The author thanks Jeffrey L. Juaregui for reading an earlier version of this paper, Denis Auroux, Graham Cox, Alexandru Oancea, and Miguel S\'anchez for helpful discussions, and the referee for numerous helpful comments, in particular for Remarks 4 and 6.

\bibliographystyle{siam}
\bibliography{contact-symplectic-FINAL}
\end{document}